\documentclass[a4paper,12pt,reqno]{amsart}
\usepackage{amsthm}
\usepackage{amssymb}
\usepackage{amsmath}
\usepackage{psfrag}
\usepackage{epsfig}
\newtheorem{theorem}{Theorem}

\newtheorem{lemma}[theorem]{Lemma}
\newtheorem{prop}[theorem]{Proposition}
\newtheorem{cor}[theorem]{Corollary}

\theoremstyle{definition}

\theoremstyle{plain}

\def\N{{\mathbb N}}

\def\eps{\varepsilon}
\def\lcm{\operatorname{lcm}}
\newcommand{\ep}[2]{\equiv{#1}\pmod{#2}}

\def\ord{\operatorname{ord}}
\def\maxx{{\max\nolimits}}
\date{}
\title{The divisibility of $a^n-b^n$  by powers of $n$}
\author{ Chris Smyth}
\address{School of Mathematics and Maxwell Institute for Mathematical Sciences, University of Edinburgh,
    James Clerk Maxwell Building, King's Buildings,
    Mayfield Road, Edinburgh EH9 3JZ, UK.}
\begin{document}
\subjclass[2000]{Primary 11B37, Secondary 11D61}
\keywords{Divisibility, Zsigmondy's Theorem, $a^n-b^n$
}

\begin{abstract}
For given integers $a,b$ and $j\ge 1$ we determine the set $R^{(j)}_{a,b}$ of integers $n$ for which $a^n-b^n$ is divisible by $n^j$. For $j=1, 2$, this set is usually infinite;  we determine  explicitly the exceptional cases for which $a,b$ the set $R^{(j)}_{a,b}\,(j=1, 2)$ is finite. For $j=2$, we use  Zsigmondy's Theorem for this.  For $j\ge 3$ and $\gcd(a,b)=1$, $R^{(j)}_{a,b}$ is probably always finite; this seems difficult to prove, however.

We also show that determination of the set of integers $n$ for which $a^n+b^n$ is divisible by $n^j$ can be reduced to that of $R^{(j)}_{a,b}$.
\end{abstract}
\maketitle
\section{Introduction}\label{S-1}
Let $a$,  $b$ and $j$ be fixed integers, with $j\ge 1$. The aim of this paper is to find the set $R^{(j)}_{a,b}$ of all positive integers $n$ such that $n^j$ divides $a^n-b^n$. For $j=1,2,\dots,$ these sets are clearly nested, with common intersection $\{1\}$. Our first  results (Theorems \ref{T-main1} and \ref{T-mainj}) describe this set in the case that $\gcd(a,b)=1$. In Section \ref{S-gto1} we describe  (Theorem \ref{T-general})  the set in the general situation where $\gcd(a,b)$ is unrestricted.

  \begin{theorem}\label{T-main1}
  Suppose that $\gcd(a,b)=1$. Then 
 the elements of the set $R^{(1)}_{a,b}$ consist of those integers $n$ whose prime factorization can be written in the form
  \begin{equation}\label{E-fact}
 n=p_1^{k_1}p_2^{k_2}\dots p_r^{k_r}\quad (p_1<p_2<\dots<p_r, \text{ all }k_i\ge 1),
 \end{equation}  
  where  $p_i\mid a^{n_i}-b^{n_i}\,(i=1,\dots,r)$, with $n_1=1$ and    $n_i=p_1^{k_1}p_2^{k_2}\dots p_{i-1}^{k_{i-1}}$ $(i=2,\dots,r)$. 
\end{theorem}   
  
In this theorem, the $k_i$ are arbitrary positive integers. The result is essentially contained in \cite{S}, which described the indices $n$ for which the generalised Fibonacci numbers $u_n$ are divisible by $n$. However, we present a self-contained proof in this paper.

 On the other hand, for $j\ge 2$, the  exponents $k_i$ are more restricted.
  \begin{theorem}\label{T-mainj}
  Suppose that $\gcd(a,b)=1$, and $j\ge 2$. Then 
 the elements of the set $R^{(j)}_{a,b}$ consist of those integers $n$ whose prime factorization can be written in the form
  (\ref{E-fact}),
   where  
   $$
   p_1^{(j-1)k_1} \text{ divides } \begin{cases} a-b  \qquad\qquad\quad\text{ if } p_1>2;\\ \lcm(a-b,a+b) \text{ if } p_1=2, \end{cases}
   $$
    and $p_i^{(j-1)k_i}\mid a^{n_i}-b^{n_i}$, with    $n_i=p_1^{k_1}p_2^{k_2}\dots p_{i-1}^{k_{i-1}}\,(i=2,\dots,r)$.

\end{theorem} 

 Thus we see that construction of $n\in R^{(j)}_{a,b}$ depends upon finding a  prime $p_i$ not used previously with $a^{n_i}-b^{n_i}$ being divisible by $p_i^{j-1}$. This presents no problem for $j=2$, so that $R^{(2)}_{a,b}$, as well as $R^{(1)}_{a,b}$, are usually infinite. See Section \ref{S-finite} for details, including the exceptional cases when they are finite. 
    However, for $j\ge 3$  the condition $p_i^{j-1}\mid a^{n_i}-b^{n_i}$ is only     rarely satisfied. This suggests strongly that in this case $R^{(j)}_{a,b}$ is always finite for $\gcd(a,b)=1$. This seems very difficult to prove, even assuming the ABC Conjecture. A result of Ribenboim and Walsh \cite{MR1670556} implies that, under ABC, the powerful part of $a^n-b^n$ cannot often be large. But this is not strong enough for what is needed here.
    On the other hand,  $R^{(j)}_{a,b}\, (j\ge 3)$ can be made arbitrarily large by choosing $a$ and $b$ such that $a-b$ is a powerful number. For instance, choosing  $a=1+(q_1q_2\dots q_s)^{j-1}$ and $b=1$, where $q_1,q_2,\dots, q_s$ are distinct primes, then $R^{(j)}_{a,b}$ contains the $2^s$ numbers $q_1^{\eps_1}q_2^{\eps_2}\dots q_s^{\eps_s}$ where the $\eps_i$ are $0$ or $1$.  See Example \ref{X-j} in Section \ref{S-exs}.

 In the next section we give preliminary results need for the proof of the theorem. We prove it in Section \ref{S-proof}. In Section \ref{S-gto1} we describe  (Theorem \ref{T-general})  $R^{(j)}_{a,b}$,  where $\gcd(a,b)$ is unrestricted. In  Section \ref{S-finite} we find all $a,b$ for which $R^{(2)}_{a,b}$ is finite (Theorem \ref{T-finite}). In  Section \ref{S-plus} we discuss the divisibility of $a^n+b^n$ by powers of $n$. In Section \ref{S-exs} we give some examples, and make some final remarks in Section \ref{S-8}.

 \section{Preliminary results}
 
We first prove a version of Fermat's Little Theorem that gives a little bit more information in the case $x\ep{1}{p}$.
\begin{lemma}\label{L-Wee} For $x\in \mathbb Z$ and $p$ an odd prime we have
 \begin{equation}\label{E-Wee}
 x^{p-1}+x^{p-2}+\dots+x+1\equiv\begin{cases} p\pmod{p^2} \text{ if } x\ep{1}{p};\\1\pmod{p} \text{ otherwise }.\end{cases}
 \end{equation}
 \end{lemma}
 \begin{proof}If $x\ep{1}{p}$, say $x=1+kp$, then $x^j\ep{1+jkp}{p^2}$, so that
 \begin{equation}\label{E-Wee2}
 x^{p-1}+x^{p-2}+\dots+x+1\equiv p+kp\sum_{j=0}^{p-1}j\ep{p}{p^2}.
 \end{equation}
 Otherwise 
  \begin{equation}\label{E-Wee3}
x(x-1)(x^{p-2}+\dots+x+1)= x^{p}-x\ep{0}{p},
 \end{equation}
 so that for $x\not\equiv{1}\pmod{p}$ we have $x(x^{p-2}+\dots+x+1)\ep{0}{p}$, and hence 
 \begin{equation}\label{E-Wee4}
 x^{p-1}+x^{p-2}+\dots+x+1\equiv x(x^{p-2}+\dots+x+1) +1\ep{1}{p}.
 \end{equation}
 \end{proof}
 
 The following is a result of Birkoff and Vandiver \cite[Theorem III]{MR1503541}. It is also special case  of Lucas \cite[p. 210]{MR1505164}, as corrected for $p=2$ by Carmichael \cite[Theorem X]{MR1517263}.
 \begin{lemma}\label{L-ab} Let $\gcd(a,b)=1$ and $p$ be prime with $p\mid a-b$. Define $t>0$ by $p^t\|a-b$ for $p>2$ and  $2^t\|\lcm(a-b,a+b)$ if $p=2$. Then for $\ell>0$
 \begin{equation}\label{E-ab}
p^{t+\ell}\|a^{p^\ell}-b^{p^\ell}.
 \end{equation}
 On the other hand,  if $p\nmid a-b$ then for $\ell\ge 0$
 \begin{equation}\label{E-ab2}
p\nmid a^{p^\ell}-b^{p^\ell}.
 \end{equation} 
 \end{lemma}
 
 \begin{proof}
 Put $x=a/b$. First suppose that $p$ is odd and $p^t\|a-b$ for some $t>0$. Then as $\gcd(a,b)=1$,  $b$ is not divisible by $p$, and we have $x\ep{1}{p^t}$.
 Then from 
\begin{equation}\label{E-ab3}
 a^p-b^p=(a-b)b^{p-1}(x^{p-1}+x^{p-2}+\dots+x+1)
 \end{equation} 
  we have by Lemma \ref{L-Wee} that $p^{t+1}\|a^p-b^p$. 
  Applying this result $\ell$ times, we obtain (\ref{E-ab}).
  
  For $p=2$, we have $p^{t+1}\|a^2-b^2$ and from $a^2\equiv b^2\ep{1}{8}$, we obtain  $2^1\|a^2+b^2$, and so $p^{t+2}\|a^4-b^4$. An easy induction then gives the required result.
 
 Now suppose that $p\nmid a-b$. Since $\gcd(a,b)=1$, (\ref{E-ab2}) clearly holds if $p\mid a$ or $p\mid b$, as must happen for $p=2$. So we can assume that $p$ is odd and $p\nmid b$. Then $x\not\equiv{1}\pmod{p}$ so that, by Lemma \ref{L-Wee} and 
 (\ref{E-ab3}), we have $p\nmid a^{p}-b^{p}$. Applying this argument $\ell$ times, we obtain (\ref{E-ab2}).
  \end{proof}

 For $n\in R^{(j)}_{a,b}$, we now define the set $\mathcal P^{(j)}_{a,b}(n)$ to be the set of all prime powers $p^k$ for which $np^k\in R^{(j)}_{a,b}$. Our next result describes this set precisely. (Compare with \cite[Theorem 1(a)]{S}).

 \begin{prop}\label{P-upj}
 Suppose that $j\ge 1$, $\gcd(a,b)=1$, $n\in R^{(j)}_{a,b}$ and
  \begin{equation}\label{E-upj}
a^n-b^n=2^{e'_2}\prod_{p>2} p^{e_p},\quad n=\prod_p p^{k_p}
 \end{equation} 
 and define $e_2$ by $2^{e_2}\|\lcm(a^n-b^n,a^n+b^n)$. Then
  \begin{equation}\label{E-p1}\mathcal P^{(1)}(n)=\bigcup_{p\mid a^n-b^n}\{p^k, k\in\N\},\end{equation} 
 and for $j\ge 2$
  \begin{equation}\label{E-pj}\mathcal P^{(j)}_{a,b}(n)=\bigcup_{p:p^{j-1}\mid a^n-b^n}\left\{p^k:1\le k\le \lfloor \frac{e_p-jk_p}{j-1}\rfloor\right\}.\end{equation}
 \end{prop}

 Note that  $e_2$ is never $1$. Consequently, if $2m\in R^{(2)}_{a,b}$, where $m$ is odd, then $4m\in R^{(2)}_{a,b}$. Also, $2\in R^{(j)}_{a,b}$ for $j\le 3$ when $a-b$ is even.
  \begin{proof}
 Taking $n\in R^{(j)}_{a,b}$ we have, from (\ref{E-upj}) and the definition of $e_2$ that $jk_p\le e_p$ for all primes $p$. Hence, applying Lemma \ref{L-ab} with $a,b$ replaced by $a^n,b^n$  we have for $p$ dividing $a^n-b^n$ that for $\ell>0$
 \begin{equation}\label{E-up2}
p^{e_p+\ell}\|a^{np^\ell}-b^{np^\ell}.
 \end{equation}  
 So  $(np^\ell)^j\mid a^{np^\ell}-b^{np^\ell}$ is equivalent to $j(k_p+\ell)\le e_p+\ell$, or $(j-1)\ell\le e_p-jk_p$. Thus we obtain (\ref{E-p1}) for $j\ge 2$, with $\ell$ unrestricted for $j=1$, giving (\ref{E-p1}).
 
 On the other hand, if $p\nmid a^n-b^n$, then by  Lemma \ref{L-ab} again, $p^\ell\nmid a^{np^\ell}-b^{np^\ell}$, so that certainly $(np^\ell)^j\nmid a^{np^\ell}-b^{np^\ell}$.
  \end{proof}

We now recall some facts about the order function $\ord$. For $m$ an integer greater than $1$ and $x$ an integer prime to $m$, we define $\ord_m(x)$, the {\it order of $x$ modulo $m$}, to be the least positive integer $h$ such that $x^h\ep{1}{m}$.
The next three lemmas, containing standard material on the $\ord$ function, are included for completeness.
 \begin{lemma}\label{L-mx}
 For $x\in \mathbb N$ and prime to $m$ we have $m\mid x^n-1$ if and only if $\ord_m(x)\mid n$.
 \end{lemma}

 \begin{proof}
 Let $\ord_m(x)=h$, and assume that $m\mid x^n-1$. Then as $m\mid x^h-1$, also $m\mid x^{\gcd(h,n)}-1$. By the minimality of $h$, $\gcd(h,n)=h$, i.e., $h\mid n$. Conversely, if $h\mid n$ then $x^h-1\mid x^n-1$, so that $m\mid x^n-1$.
  \end{proof}
 
\begin{cor}\label{C-mx} Let $j\ge 1$. We have $n^j\mid x^n
-1$ if and only if $\gcd(x,n)=1$ and $\ord_{n^j
}(x)\mid n$.
\end{cor}

 \begin{lemma}\label{L-lcm}
 For $m=\prod_p p^{f_p}$ and $x\in \mathbb N$ and prime to $m$ we have
  \begin{equation}\label{E-lcm}
\ord_m(x)=\lcm_p\ord_{p^{k_p}}(x).
 \end{equation} 
 \end{lemma}

 \begin{proof}
Put $h_p=\ord_{p^{f_p}}(x)$, $h=\ord_m(x)$ and $h'=\lcm_p h_p$. Then by Lemma \ref{L-mx} we have $p^{f_p}\mid x^{h'}-1$ for all $p$, and hence $m\mid x^{h'}-1$. Hence $h \mid {h'}$. On the other hand, as $p^{f_p}\mid  n$ and $m\mid x^h-1$, we have $p^{f_p}\mid  x^h-1$, and so $h_p\mid h$, by Lemma \ref{L-mx}. Hence $h'=\lcm_p h_p\mid h$.
  \end{proof}

Now put $p_*=\ord_p(x)$, and define $t>0$ by $p^t\| x^{p_*}-1$.

 \begin{lemma}\label{L-ord}
For $\gcd(x,n)=1$ and $\ell>0$ we have  $p_*\mid p-1$ and $\ord_{p^\ell}(x)= p^{\max(\ell-t,0)} p_*$.
 \end{lemma}

 \begin{proof}
Since $p\mid x^{p-1}-1$, we have $p_*\mid p-1$, by Lemma \ref{L-mx}. 
Also, from $p^{\ell}\mid x^{\ord_{p^{\ell}}(x)}-1$ we have $p\mid x^{\ord_{p^{\ell}}(x)}-1$, and so, by  Lemma \ref{L-mx} again, $p_*=\ord_p(x)\mid \ord_{p^{\ell}}(x)$. Further, if $\ell\le t$ then from  $p^\ell\mid  x^{p_*}-1$ we have by Lemma \ref{L-mx} that $\ord_{p^\ell}(x)\mid p_*$, so $\ord_{p^\ell}(x)=p_*$.
Further,  by Lemma \ref{L-ab} for $u\ge t$ 
 \begin{equation}\label{E-tu}
 p^{u}\| x^{p^{u-t} p_*}-1,
\end{equation}  
  so that, taking $u=\ell\ge t$ and using Lemma \ref{L-mx}, $\ord_{p^{\ell}}(x)\mid p^{\ell-t} p_*$.
 Also, if $t\le u<\ell$, then, from (\ref{E-tu}), $x^{p^{t-u}p_*}\not\equiv{1}\pmod{p^{\ell}}$. Hence $\ord_{p^{\ell}}(x)= p^{\ell-t} p_*$ for $\ell\ge t$.
  \end{proof}

\begin{cor}\label{C-mm} 
Let $j\ge 1$. For $n=\prod_p p^{k_p}$ and $x\in \mathbb N$ and prime to $n$ we have $n^j\mid x^n-1$ if and only if $\gcd(x,n)=1$ and 
\begin{equation}\label{E-mm}
 \lcm_p p^{k'_p}p_*\mid \prod_p p^{k_p}.
\end{equation}   Here the $k'_p= \max(jk_p-t_p,0)$ are integers with  $t_p>0$.
\end{cor}
Note that $p_*$, $k'_p$ and $t_p$ in general depend on $x$ and $j$ as well as on $p$.

What we actually need in our situation is the following variant of Corollary \ref{C-mm}.
\begin{cor}\label{C-mmab} 
Let $j\ge 1$. For $n=\prod_p p^{k_p}$ and integers $a,b$ with $\gcd(a,b)=1$ we have $n^j\mid a^n-b^n$ if and only if $\gcd(n,a)=\gcd(n,b)=1$ and 
\begin{equation}\label{E-mmab}
 \lcm_p p^{k'_p}p_*\mid \prod_p p^{k_p}.
\end{equation}   Here the $k'_p= \max(jk_p-t_p,0)$ are integers with  $t_p>0$.
\end{cor}
This corollary is easily deduced from the previous one by choosing $x$ with $bx\ep{a}{n^j}$.

 By contrast with Proposition \ref{P-upj}, our next proposition allows us to {\it divide} an element $n\in R^{(j)}_{a,b}$ by a prime, and remain within $R^{(j)}_{a,b}$.
 
 \begin{prop}\label{P-down}
 Let $n\in R^{(j)}_{a,b}$ with $n>1$, and suppose that $p_{\maxx}$ is the largest prime factor of $n$. Then $n/{p_\maxx}\in R^{(j)}_{a,b}$.
 \end{prop} 
  \begin{proof}
  Suppose $n\in R^{(j)}_{a,b}$, so that (\ref{E-mm}) holds, with $x=a/b$, and put $q=p_\maxx$. Then, since for every $p$ all prime factors of  $p_*$ are less than $p$, the only possible term on the left-hand side that divides $q^{k_{q}}$ on the right-hand side is the term $q^{k'_{q}}$. Now reducing $k_{q}$ by $1$ will  reduce $k'_{q}$ by at least $1$, unless it is already $0$, when it does not change. In either case (\ref{E-mm}) will still hold with $n$ replaced by $n/{q}$, and so $n/{q}\in R^{(j)}_{a,b}$.
   \end{proof}

 Various versions and special cases of Proposition \ref{P-down} for $j=1$ have been known for some time, in the more general setting of Lucas sequences, due to  Somer \cite[Theorem 5(iv)]{MR1271392}, Jarden \cite[Theorem E]{MR0197383}, Hoggatt and
Bergum \cite{MR0349567}, Walsh \cite{Wa}, Andr\'e-Jeannin \cite{MR1131414} 
and  others. See also Smyth \cite[Theorem 3]{S}. 

   In order to work out for which $a,b$ the set $R^{(j)}_{a,b}$ is finite, we need the following classical result. Recall that $a^n-b^n$ is said to have a {\it primitive prime divisor} $p$ if the prime $p$ divides $a^n-b^n$ but does not divide $a^k-b^k$ for any $k$ with $1\le k<n$.

   \begin{theorem}[{{Zsigmondy \cite{MR1271392}}}]\label{S-Z} Suppose that $a$ and $b$  are  nonzero coprime integers with $a>b$ and $a+b>0$. Then, except when
   \begin{itemize}
   \item $n=2$ and $a+b$ is a power of $2$
   
   \noindent or
   \item $n=3$, $a=2$, $b=-1$
   
   \noindent or
   \item $n=6$, $a=2$, $b=1$,
   \end{itemize}
 $a^n-b^n$ has a primitive prime divisor.  
 \end{theorem}
 (Note that in this statement we have allowed $b$ to be negative, as did Zsigmony. His theorem is nowadays often quoted with the restriction $a>b>0$ and so has the second exceptional case omitted.)

 \section{Proof of Theorems \ref{T-main1} and \ref{T-mainj}}\label{S-proof} Let $n\in R^{(j)}_{a,b}$ have a factorisation (\ref{E-fact}), where $p_1<p_2<\dots<p_r$ and all $k_i>0$. First take $j\ge 1$. Then by Proposition \ref{P-down} $n/{p_r^{k_r}}=n_{r}\in R^{(j)}_{a,b}$, and hence $$(n/{p_r^{k_r}})/{p_{r-1}^{k_{r-1}}}=n_{r-1},\quad \dots,\quad  p_1^{k_1}=n_2,\quad 1=n_1$$ are 
 all in $R^{(j)}_{a,b}$.
 Now separate the two cases $j=1$ and $j\ge 2$ for  Theorems \ref{T-main1} and \ref{T-mainj} respectively. Now for $j=1$ Proposition \ref{P-upj} gives us that $p_i\mid a^{n_i}-b^{n_i}\,(i=1,\dots,r)$, while for $j\ge 2$ we have, again from Proposition \ref{P-upj}, that
 $$
   p_1^{(j-1)k_1} \text{ divides } \begin{cases} a-b  \qquad\qquad\quad\text{ if } p_1>2;\\ \lcm(a-b,a+b) \text{ if } p_1=2, \end{cases}
   $$ 
   and  
   $p_i^{(j-1)k_i}\mid a^{n_i}-b^{n_i}\, (i=2,\dots,r)$.
 Here we have used the fact that $\gcd(p_i,n_i)=1$, so that if $p_i^{k_i}\mid   (a^{n_i}-b^{n_i})/n_i^2$ then $p_i^{k_i}\mid a^{n_i}-b^{n_i}$ (i.e., we are applying Proposition \ref{P-upj} with all the exponents $k_p$ equal to $0$.)

  \section{Finding  $R^{(j)}_{a,b}$ when $\gcd(a,b)>1$.}\label{S-gto1}  
  
 For $a>1$, define the set $\mathcal F_a$ to be the set of all $n\in \mathcal N$ whose prime factors all divide $a$. To find $R^{(j)}_{a,b}$ in general, we first consider the case $b=0$.
 
 \begin{prop}\label{P-Ra0}
 We have $R^{(1)}_{a,0}=R^{(2)}_{a,0}=\mathcal F_a$, while for $j\ge 3$ the set $R^{(j)}_{a,0}=\mathcal F_a\setminus S^{(j)}_a$, where $S^{(j)}_a$ is a finite set.
 \end{prop}
 \begin{proof}
 From the condition $n^j\mid a^n$, all prime factors of $n$ divide $a$, so 
 $R^{(j)}_{a,0}\subset\mathcal F_a$, say $R^{(j)}_{a,0}=\mathcal F_a\setminus S^{(j)}_a$. We need to prove that $S^{(j)}_a$ is finite. Suppose that $a=p_1^{a_1}\dots p_r^{a_r}$, with $p_1$ the smallest prime factor of $a$. Then $n=p_1^{k_1}\dots p_r^{k_r}$ for some $k_i\ge 0$. From $n^j\mid a^n$ we have 
 \begin{equation}\label{E-R1}
 k_i\le \frac{a_i}{j}p_1^{k_1}\dots p_r^{k_r}\quad(i=1,\dots,r).
 \end{equation}
For these $r$ conditions to be satisfied it is sufficient that 
\begin{equation}\label{E-R2}
 \sum_{i=1}^r k_i\le \frac{\min_{i=1}^r a_i}{j}p_1^{\sum_{i=1}^r k_i}.
 \end{equation}
 Now (\ref{E-R2}) holds if $j=1$ or $2$, as in this case, from the simple inequality $k\le 2^{k-1}$ valid for all $k\in\mathbb N$, we have
 \begin{equation}\label{E-R3}
 \sum_{i=1}^r k_i\le \frac12 2^{\sum_{i=1}^r k_i}\le \frac{\min_{i=1}^r a_i}{j}p_1^{\sum_{i=1}^r k_i}.
 \end{equation} 
 Hence $S^{(j)}_a$ is empty if $j=1$ or $2$. 
 
 Now take $j\ge 3$, and let $K=K^{(j)}_a$ be the smallest integer such that $Kp_1^{-K}\le 
 (\min_{i=1}^r a_i)/j$. Then (\ref{E-R2}) holds for $\sum_{i=1}^r k_i\ge K$, and $S^{(j)}_a$ is contained in the finite set $S''=\{n\in \mathbb N, n=p_1^{k_1}\dots p_r^{k_r}: \sum_{i=1}^r k_i<K\}$. (To compute $S^{(j)}_a$ precisely, one need just check for which $r$-tuples $(k_1,\dots,k_r)$ with 
 $\sum_{i=1}^r k_i<K$ any of the $r$ inequalities of (\ref{E-R1}) is violated.
 \end{proof}
 
 One (at first sight) curious consequence of the equality $R^{(1)}_{a,0}=R^{(2)}_{a,0}$ above is that $n\mid a^n$ implies $n^2\mid a^n$.

 Now let $g=\gcd(a,b)$ and $a=a_1g$, $b=b_1g$. Write $n=Gn_1$, where all prime factors of $G$ divide $g$ and $\gcd(n_1,g)=1$. Then we have the following general result.
 
 \begin{theorem}\label{T-general} The set $R^{(j)}_{a,b}$ is given by
 \begin{equation}\label{E-g4}
   R^{(j)}_{a,b}=\{n=Gn_1: G\in\mathcal F_g, n_1\in R^{(j)}_{a^G_1,b^G_1} \text{ and } \gcd(g,n_1)=1\}\setminus R,
    \end{equation}
  where $R$ is a finite set. Specifically, all $n=Gn_1\in R$ have $1\le n_1<j/2$ and \begin{equation}\label{E-q1}
  G=q_1^{\ell_1}\dots q_m^{\ell_m},
  \end{equation}
  where 
  \begin{equation}\label{E-q2}
  \sum_{i=1}^m\ell_i<K^{(j)}_{g^{n_1}}.
  \end{equation} 
  Here the $q_i$ are the primes dividing $g$, and $K^{(j)}_{g^{n_1}}$ is the constant in the proof of Proposition \ref{P-Ra0} above.
\end{theorem} 
 \begin{proof}
 supposing that $n\in R^{(j)}_{a,b}$ we have
 \begin{equation}\label{E-n2}
 n^j\mid a^n-b^n
 \end{equation}
 and so $n^j\mid g^n(a_1^n-b_1^n)$. Writing $n=Gn_1$, as above, we have 
      \begin{equation}\label{E-red1}
    n_1^j\mid (a_1^G)^{n_1}-(b_1^G)^{n_1}
    \end{equation}
     and 
     \begin{equation}\label{E-red2}
     G^j\mid g^{Gn_1}\left((a_1^G)^{n_1}-(b_1^G)^{n_1}\right).
     \end{equation}
     Thus (\ref{E-n2}) holds with $n,a,b$ replaced by $n_1,a_1^G,b_1^G$. So we have reduced the problem of (\ref{E-n2}) to a case where $\gcd(a,b)=1$, which we can solve for $n_1$ prime to $g$, along with the extra condition (\ref{E-red2}). Now, from the fact that $R^{(2)}_{g,0}=\mathcal F_g$
 from Proposition \ref{P-Ra0}, we have $G^2\mid g^G$ and hence $G^j\mid g^{Gn_1}$ for all $G\in \mathcal F_g$ , provided that $n_1\ge j/2$. Hence (\ref{E-red2}) can fail to hold for all $G\in \mathcal F_g$ only for $1\le n_1<j/2$. 
 
 Now fix $n_1$ with  $1\le n_1<j/2$. Then note that  by Proposition \ref{P-Ra0}, $G^j\mid g^{Gn_1}$ and hence (\ref{E-n2}) holds for all $G\in\mathcal F_{g^{n_1}}\setminus S$, where $S$ is a finite set of $G$'s contained in the set of all $G$'s given by (\ref{E-q1}) and (\ref{E-q2}).
 \end{proof}

 Note that (taking $n_1=1$ and using (\ref{E-red2})) we always have $R^{(j)}_{g,0}\subset R^{(j)}_{a,b}$. See example \ref{X-gcd2} in Section \ref{S-exs}.

 \section{When are $R^{(1)}_{a,b}$ and $R^{(2)}_{a,b}$ finite?}\label{S-finite}

First consider $R^{(1)}_{a,b}$. From Theorem \ref{T-main1} it is immediate that $R^{(1)}_{a,b}$  contains all powers of any primes dividing $a-b$. Thus $R^{(1)}_{a,b}$  is infinite unless $a-b=\pm 1$, in which case $R^{(1)}_{a,b}=\{1\}.$ This was pointed out earlier by  Andr{\'e}-Jeannin \cite[Corollary 4]{MR1131414}.

Next, take $j=2$.  Let us denote by $\mathcal P^{(2)}_{a,b}$ the set of primes that divide some $n\in R^{(2)}_{a,b}$ and, as before, put $g=\gcd(a,b)$.
 \begin{theorem}\label{T-finite} The set $R^{(2)}_{a,b}=\{1\}$ if and only if $a$ and $b$ are consecutive integers, and $R^{(2)}_{a,b}=\{1,3\}$ if and only if $ab=-2$. Otherwise, $R^{(2)}_{a,b}$ is infinite.
 
 If  $R^{(2)}_{a/g,b/g}=\{1\}$ (respectively, $=\{1,3\}$) then 
  $\mathcal P^{(2)}_{a,b}$ is the set of all prime divisors of $g$ (respectively, $3g$). Otherwise $\mathcal P^{(2)}_{a,b}$ is infinite.
 \end{theorem}

The application of  Zsigmondy's Theorem that we require is the following.

\begin{prop} \label{P-Zap}If $R^{(2)}_{a,b}$ contains some integer $n\ge 4$ then both $R^{(2)}_{a,b}$ and $\mathcal P^{(2)}_{a,b}$ are infinite sets.
\end{prop}
\begin{proof} First note that if $a=2$, $b=1$ (or more generally $a-b=\pm 1$) then by Theorem \ref{T-mainj}, $R^{(2)}=\{1\}$. Hence, taking $n\in R^{(2)}_{a,b}$ with $n\ge 4$ we have, by Zsigmondy's Theorem, that $a^n-b^n$ has a primitive prime divisor, $p$ say. Now if $p\mid n$ then, by applying Proposition \ref{P-down} as many times as necessary we find $p\mid n'$, where $n'\in R^{(2)}_{a,b}$ and now $p$ is the maximal prime divisor of $n'$. Hence, by Proposition \ref{P-down} again, ${n''}=n'/p\in R^{(2)}_{a,b}$ and so, from $n'=p{n''}$ and Proposition \ref{P-upj} we have that $p\mid a^{{n''}}-b^{{n''}}$, contradicting the primitivity of $p$. 

Now using Proposition \ref{P-upj} again, $np\in R^{(2)}_{a,b}$. Repeating the argument with $n$ replaced by $np$ and continuing in this way we obtain an infinite sequence $n, np, npp_1,npp_1p_2, \dots, npp_1p_2\dots p_\ell, \dots$ of elements of $R^{(2)}_{a,b}$, where $p<p_1<p_2<\dots <p_\ell<\dots$ are primes.
\end{proof}

   \begin{proof}[Proof of Theorem \ref{T-finite}]
Assume $\gcd(a,b)=1$, and, without loss of generality, that $a>0$ and $a>b$.
(We can ensure this by interchanging $a$ and $b$ and/or changing both their signs.)
If $a-b$ is even, then $a$ and $b$ are odd, and $a^2-b^2\ep{1}{2^{t+1}}$, where $t\ge 2$. Hence $4\in R^{(2)}_{a,b}$, by  Proposition \ref{P-upj}, and so both $R^{(2)}_{a,b}$ and $\mathcal P^{(2)}_{a,b}$ are infinite sets, by Proposition \ref{P-Zap}.

 If $a-b=1$  then $R^{(2)}=\{1\}$, as we have just seen, above.

If $a-b$ is odd and  at least $5$, then $a-b$ must either be divisible by $9$ or by a prime  $p\ge 5$. Hence $9$ or $p$ belong to $R^{(2)}_{a,b}$, by Proposition \ref{P-upj}, and again both $R^{(2)}_{a,b}$ and $\mathcal P^{(2)}_{a,b}$ are infinite sets, by Proposition \ref{P-Zap}.

If $a-b= 3$ then $3\in R^{(2)}_{a,b}$, and $a^3-b^3=9(b^2+3b+3)$. If $b=-1$ (and $a=2$, $ab=-2$) or $-2$ (and $a=1$, $ab=-2$) then   $a^3-b^3=9$ and so, by Theorem \ref{T-mainj}, so $R^{(2)}=\{1,3\}$. Otherwise, using $\gcd(a,b)=1$ we see that $a^3-b^3\ge 5$, and so the argument for $a-b\ge 5$ but with $a,b$ replaced by $a^3,b^3$ applies.
\end{proof}

   \section{The powers of $n$ dividing  $a^n+b^n$}\label{S-plus}
   
   Define $R^{(j)+}_{a,b}$ to be the set $\{n\in\N: n^j \text{ divides } a^n+b^n\}$.
   Take $j\ge 1$, and assume that  $\gcd(a,b)=1$. (The general case $\gcd(a,b)\ge 1$ can be handled as in Section \ref{S-gto1}.) 
    We then have the following result.
    
 \begin{theorem}\label{T-plus} Suppose that $j\ge 1$, $\gcd(a,b)=1$, $a>0$ and $a\ge|b|$. Then   \begin{itemize}
   \item[(a)]  $R^{(1)+}_{a,b}$ consists of the odd elements of $R^{(1)}_{a,-b}$, along with the numbers of the form $2n_1$, where $n_1$ is an odd element of $R^{(1)}_{a^2,-b^2}$;
\item[(b)] If $j\ge 2$  the set $R^{(j)+}_{a,b}$ consists  of odd elements of $R^{(j)}_{a,-b}$ only . 
\end{itemize}
Furthermore, for $j=1$ and $2$, the set $R^{(j)+}_{a,b}$ is infinite, except in the following cases:
  \begin{itemize}
   \item If  $a+b$ is $1$ or  a power of $2$, $(j,a,b)\ne (1,1,1)$, when it is $\{1\}$;
\item $R^{(1)+}_{1,1}=\{1,2\}$;
\item $R^{(2)+}_{2,1}=\{1,3\}$.
\end{itemize}
 \end{theorem}

  \begin{proof}
      If $n$ is even and $j\ge 2$, or if $4\mid n$ and $j=1$, then $n^j\mid a^n+b^n$ implies that $4\mid a^n+b^n$, contradicting the fact that,  as $a$ and $b$ are not both even, $a^n+b^n\equiv{1}$ or $2\pmod{8}$. So either
   \begin{itemize}
   \item  $n$ is odd, in which case $n^j\mid a^n+b^n$ is equivalent to finding the odd elements of the set $R^{(j)}_{a,-b}$;
   
   \noindent\noindent or
   \item $j=1$ and $n=2n_1$, where $n_1$ is odd, and belongs to $R^{(1)}_{a^2,-b^2}$. 
  \end{itemize} 
   Now suppose that $j=1$ or $2$. If $a+b$ is $\pm 1$ or $\pm$ a power of $2$, then, by Theorem \ref{T-mainj}, all $n\in   R^{(j)}_{a,-b}$ with $n>1$ are even, so for $j= 2$ there are no $n>1$ with $n^j\mid a^n+b^n$ in this case. Otherwise, $a+b$ will have an odd prime factor, and so at least one odd element $>1$. By Theorem \ref{T-finite} and its proof, we see that $R^{(2)}_{a,-b}$ will have infinitely many odd elements unless $a(-b)=-2$, i.e. $a=2$, $b=1$ (using $a>0$ and $a\ge|b|$).

   For $j=1$, there will be infinitely many  $n$ with $n\mid a^n+b^n$, except when both $a+b$ and $a^2+b^2$ are $1$ or  a power of $2$. It is an easy exercise to check that,  this can happen only for $a=b=1$ or $a=1$, $b=0$.
     \end{proof}
     
     If $g=\gcd(a,b)>1$, then, since $R^{(j)+}_{a,b}$ contains the set 
  $R^{(j)}_{g,0}$, it will be infinite, by Proposition \ref{P-Ra0}.     
 For $j\ge 3$ and $\gcd(a,b)=1$, the finiteness of the set $R^{(j)+}_{a,b}$ would follow from the finiteness of $R^{(j)}_{a,b}$, using Theorem \ref{T-finite}(b).

\section{Examples.}\label{S-exs}
The set $R^{(j)}_{a,b}$ has a natural labelled, directed-graph structure, as follows: take the vertices to be the elements of $R^{(j)}_{a,b}$, and join a vertex $n$ to a vertex $np$ as $n\to_p np$, where $p\in\mathcal P^{(j)}_{a,b}$. We reduce this to a spanning tree of this graph by taking only those edges $n\to_p np$ for which $p$ is the largest prime factor of $np$. For our first example we draw this tree (Figure \ref{F-31}). 
\begin{enumerate}
\item[1.] Consider the set
\begin{align*}R^{(2)}_{3,1}=&1, 2, 4, 20, 220, 1220, 2420, 5060, 13420, 14740, 23620, 55660,\\ & 145420, 147620, 162140, 237820, 259820, 290620, 308660, \\ &339020, 447740, 847220, 899140, 1210220, \dots,\end{align*}
(sequence A127103 in Neil Sloane's Integer Sequences website). Now 
$$3^{20}-1=2^4\cdot 5^2\cdot 11^2\cdot 61 \cdot 1181,$$
showing that $\mathcal P^{(2)}_{3,1}(20)=\{11,11^2,61,1181\}$. Also   
\begin{align*}& 3^{220}-1=      2^4\cdot  5^3\cdot 11^3\cdot  23\cdot 61\cdot 67\cdot 661\cdot 1181\cdot 1321\cdot 3851\cdot 5501\\ 
&\cdot 177101\cdot 570461\cdot 659671\cdot 24472341743191\cdot 560088668384411\\
 &\cdot 927319729649066047885192700193701,\end{align*}
so that the elements of $\mathcal P^{(2)}_{3,1}(220)$ less than $10^6/220$, needed for Figure \ref{F-31}, are
$$11, 23, 61,  661, 1181, 1321, 3851.$$
\begin{figure}[h]
\begin{center}
\leavevmode \hbox{ \epsfxsize=4.5in
\epsffile{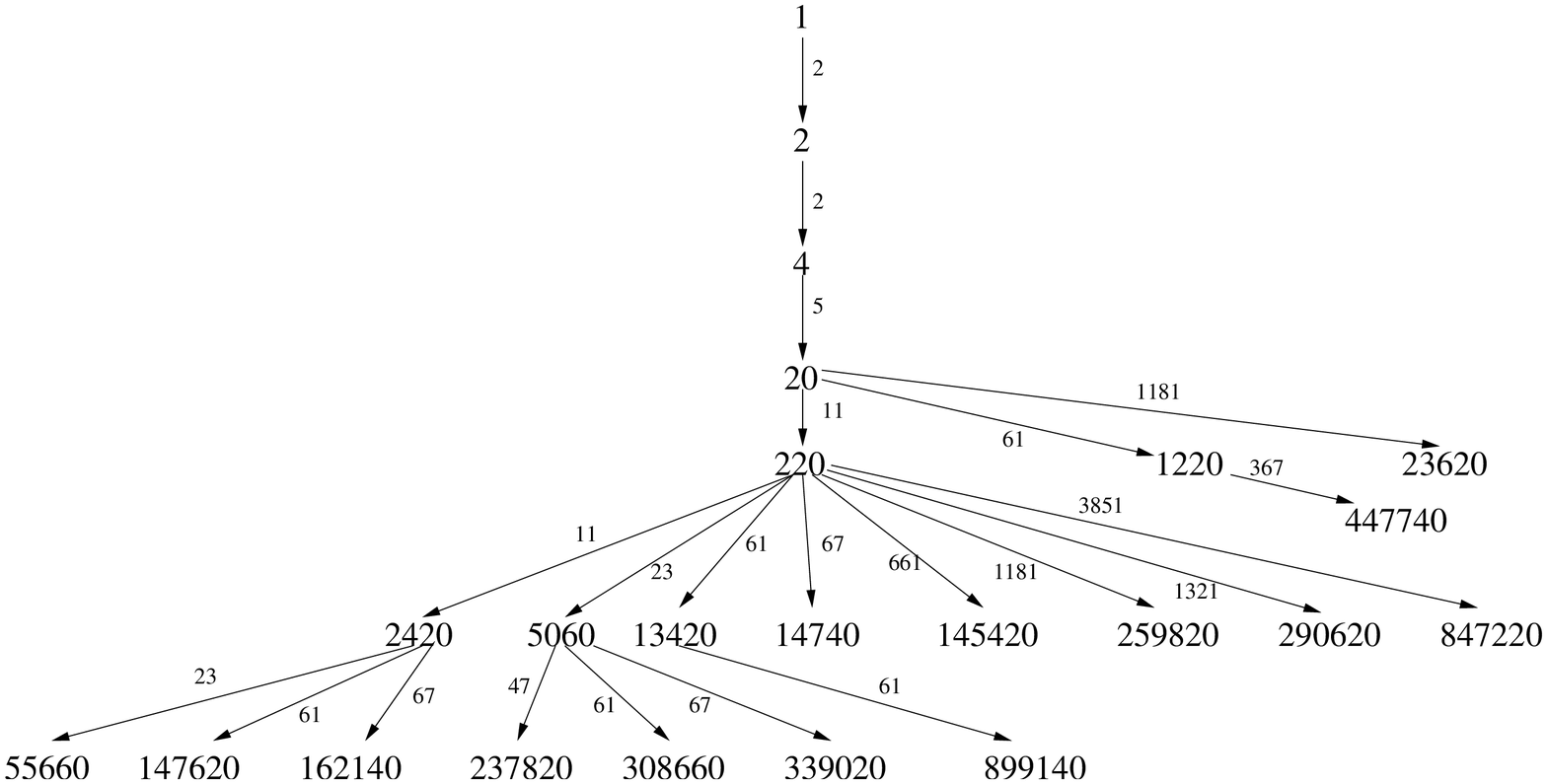} }
\end{center}
\caption{Part of the tree for $R^{(2)}_{3,1}$, showing all elements below $10^6$. } \label{F-31}
\end{figure}

\item[2.] Now $$R^{(2)}_{5,-1}
=1,2,3,4,6,12,21,42,52,84,156,186,372,\dots,$$
whose odd elements give
 $$R^{(2)+}_{5,-1}=1,3,21,609,903,2667,9429,26187,\dots.$$
See Section \ref{S-plus}.

\item[3.] We have
$$R^{(2)+}_{3,2}=  R^{(2)}_{3,-2}=1,5,55,1971145,\dots,$$
as all elements of $R^{(2)}_{3,-2}$ are odd.
Although this set is infinite by Theorem \ref{T-finite}, the next term is $1971145p$ where $p$ is the smallest prime factor of $3^{1971145}+2^{1971145}$ not dividing $1971145$. This looks difficult to compute, as it could be very large.

\item[4.] We have
$$R^{(2)}_{4,-3}=R^{(2)+}_{4,3}=1,7,2653,\dots.$$
Again, this set is infinite, but here only the three terms given are readily computable. The next term is $2653p$ where $p$ is the smallest prime factor of $4^{2653}+3^{2653}$ not dividing $2653$.

\item[5.]   This is an example of a set where more than one odd prime occurs as a squared factor in  elements of the set, in this case the primes $3$ and $7$. Every element greater than $9$ is of one of the forms $21m$, $63m$, $147m$, or $441m$, where $m$ is prime to $21$.
\begin{align*}R^{(2)}_{11,2}=&1, 3, 9, 21, 63, 147, 441, 609, 1827, 4137, 4263, 7959,\\
& 8001, 12411, 12789, 23877, 28959, 35931, 55713, 56007, \\
 &86877, 107793, 119973, 167139, 212541, 216237, 230811,\\
  &232029, 251517, 359919, 389403,\dots,
  \end{align*}
  
  \item[6.]\label{X-j} $R^{(4)}_{27001,1}=\{1, 2, 3, 5, 6, 10, 15, 30\}$. This is because $27001-1=2^3\cdot 3^3\cdot 5^3$, and none of $27001^n-1$ has a factor $p^3$ for any prime $p>5$ for any $n=1,2,3,5,6,10,15,30$.
  
  \item[7.] $R^{(3)}_{19,1}=\{1,2,3,6,42,1806\}$?
Is this the entire set? Yes, unless $19^{1806}-1$ is divisible by $p^2$ for some prime $p$ prime to $1806$, in which case $1806p$ would also be in the set. But determining whether or not this is the case seems to be a hard computational problem.

\item[8.] \label{X-gcd2} $R^{(4)}_{56,2}$,  an example with $\gcd(a,b)>1$. It seems highly probable that 
\begin{align*}
R^{(4)}_{56,2} &=(\mathcal F_2\setminus\{2,4,8\})\cup (3\mathcal F_2)\\
&= 1, 3, 6, 12, 16, 24, 32, 48, 64, 96, 128, 192, 256, 384, 512, 768, 1024,   \dots .
\end{align*}
However, in order to prove this, Theorem \ref{T-general} tells us that we need to know that $28^{2^\ell}\ne 1\pmod{p^3}$ for every  prime $p>3$ and every $\ell>0$. This seems very difficult!
Note that $R^{(4)}_{2,0}=\mathcal F_2\setminus\{2,4,8\}$ and $R^{(4)}_{28,1}=\{1,3\}.$
\end{enumerate}

\section{Final remarks.}\label{S-8}\begin{enumerate}
\item[1.]  By finding $R^{(j)}_{a,b}$, one is essentially solving the exponential Diophantine equation $x^jy=a^x-b^x$, since any solutions with $x\le 0$ are readily found.

\item[2.] It is known that $$R^{(1)}_{a,b}=\{n\in\N: n \text{ divides } \frac{a^n-b^n}{a-b}\}.$$
See \cite[Proposition 12]{S} (and also Andr{\'e}-Jeannin \cite[Theorem 2]{MR1131414} for some special cases of this result.) This result shows that $R^{(1)}_{a,b}=\{n\in\N: n \text{ divides } u_n\}$, where the $u_n$ are the generalised Fibonacci numbers of the first kind defined by the recurrence $u_0=1$, $u_1=1$, and  $u_{n+2}=(a+b)u_{n+1}-abu_n\,(n\ge 0)$. This provides a link between Theorem \ref{T-main1} of the present paper and the results of \cite{S}.

 The set $R^{(1)+}_{a,b}$ is a special case of a set $\{n\in\N: n \text{ divides } v_n\}$, also studied in \cite{S}. Here $(v_n)$ is the sequence of generalised Fibonacci numbers of the second kind. For earlier work on this topic see Somer \cite{MR1393479}.
  
  \item[3.] {\it Earlier and related work.}\label{R-earlier}
    The study of factors of $a^n-b^n$ dates back at least to Euler, who proved that all primitive prime factors of $a^n-b^n$ were $\ep{1}{n}$. See \cite[Theorem 1]{MR1503541}. Chapter 16 of Dickson \cite{MR0245499} (Vol 1) is devoted to the literature on factors of $a^n\pm b^n$.
  
    More specifically, Kennedy and Cooper \cite{MR995562} studied the set $R^{(1)}_{10,1}$. Andr{\'e}-Jeannin \cite[Corollary 4]{MR1131414} claimed (erroneously -- see Theorem \ref{T-plus})  that the congruence $a^n+b^n\ep{0}{n}$ always has infinitely many solutions $n$ for $\gcd(a,b)=1$.

\item[4.] {\it Acknowledgement.} I thank Hugh Montgomery for telling me about Zsigmondy's Theorem.
\end{enumerate}   
\bibliographystyle{plain}
\bibliography{subscripts,unpub}

\begin{thebibliography}{10}

\bibitem{MR1131414}
Richard Andr{\'e}-Jeannin.
\newblock Divisibility of generalized {F}ibonacci and {L}ucas numbers by their
  subscripts.
\newblock {\em Fibonacci Quart.}, 29(4):364--366, 1991.

\bibitem{MR1503541}
Geo.~D. Birkhoff and H.~S. Vandiver.
\newblock On the integral divisors of {$a\sp n-b\sp n$}.
\newblock {\em Ann. of Math. (2)}, 5(4):173--180, 1904.

\bibitem{MR1517263}
R.~D. Carmichael.
\newblock On the {N}umerical {F}actors of {C}ertain {A}rithmetic {F}orms.
\newblock {\em Amer. Math. Monthly}, 16(10):153--159, 1909.

\bibitem{MR0245499}
Leonard~Eugene Dickson.
\newblock {\em History of the theory of numbers. {V}ol. {I}: {D}ivisibility and
  primality.}
\newblock Chelsea Publishing Co., New York, 1966.

\bibitem{MR0349567}
Verner~E. Hoggatt, Jr. and Gerald~E. Bergum.
\newblock Divisibility and congruence relations.
\newblock {\em Fibonacci Quart.}, 12:189--195, 1974.

\bibitem{MR0197383}
Dov Jarden.
\newblock Divisibility of {F}ibonacci and {L}ucas numbers by their subscripts.
\newblock In {\em Recurring sequences: {A} collection of papers}, Second
  edition. Revised and enlarged, pages 68--75. Riveon Lematematika, Jerusalem
  (Israel), 1966.

\bibitem{MR995562}
Robert~E. Kennedy and Curtis~N. Cooper.
\newblock Niven repunits and {$10\sp n\equiv 1\pmod n$}.
\newblock {\em Fibonacci Quart.}, 27(2):139--143, 1989.

\bibitem{MR1505164}
Edouard Lucas.
\newblock Th\'eorie des {F}onctions {N}um\'eriques {S}implement
  {P}\'eriodiques.
\newblock {\em Amer. J. Math.}, 1(3):197--240, 1878.

\bibitem{MR1670556}
Paulo Ribenboim and Gary Walsh.
\newblock The {$ABC$} conjecture and the powerful part of terms in binary
  recurring sequences.
\newblock {\em J. Number Theory}, 74(1):134--147, 1999.

\bibitem{S}
Chris Smyth.
\newblock The terms in {L}ucas sequences divisible by their indices.
\newblock On arXiv: 0908.3832v1[math.NT], 2009.

\bibitem{MR1271392}
Lawrence Somer.
\newblock Divisibility of terms in {L}ucas sequences by their subscripts.
\newblock In {\em Applications of {F}ibonacci numbers, {V}ol. 5 ({S}t.
  {A}ndrews, 1992)}, pages 515--525. Kluwer Acad. Publ., Dordrecht, 1993.

\bibitem{MR1393479}
Lawrence Somer.
\newblock Divisibility of terms in {L}ucas sequences of the second kind by
  their subscripts.
\newblock In {\em Applications of {F}ibonacci numbers, {V}ol.\ 6 ({P}ullman,
  {WA}, 1994)}, pages 473--486. Kluwer Acad. Publ., Dordrecht, 1996.

\bibitem{Wa}
Gary Walsh.
\newblock On integers $n$ with the property $n\mid f_n$.
\newblock 5pp., unpublished, 1986.

\end{thebibliography}

\end{document}